\documentclass[10pt]{article}
\usepackage{}
\usepackage{mathrsfs}
\usepackage{amssymb}
\usepackage{amsmath}
\usepackage[all]{xy}
\usepackage{amsthm}
\usepackage{amsfonts}
\usepackage{latexsym}
\usepackage{amscd,amssymb,amsopn,amsmath,amsthm,graphics,amsfonts,mathrsfs,accents,enumerate,verbatim,calc}
\usepackage[dvips]{graphicx}
\usepackage[colorlinks=true,linkcolor=red,citecolor=blue]{hyperref}

\usepackage{stackrel} 

\newtheorem{thm}{Theorem}[section]
\newtheorem{prop}[thm]{Proposition}

\newtheorem{lem}[thm]{Lemma}

\newtheorem{df}[thm]{Definition}

\newtheorem{rmk}[thm]{Remark}

\newtheorem{Para}[thm]{}

\def\Hom{\mathop{\rm Hom}\nolimits}

\begin{document}

\title{\large \bf Quotients of exact categories by pseudo-cluster tilting subcategories
}
\author{Jie Xu, Yuefei Zheng \footnote{Corresponding author} \\
\it\footnotesize College of Science, Northwest A$\&$F University, Yangling 712100, Shaanxi Province, China\\
\it\footnotesize Email addresses: xujiejie@nwafu.edu.cn; zhengyf@nwafu.edu.cn
}
\date{}
\baselineskip=15pt
\maketitle
\begin{abstract}
We introduce the concept of a pseudo-cluster tilting subcategory from the viewpoint of the fact that the quotient of an exact category by a cluster tilting subcategory is an abelian category. We prove that the quotients in the case of pseudo-cluster tilting are always semi-abelian. In addition, it is abelian if and only if some self-orthogonal conditions are satisfied. We revisit the abelian quotient category of conflations by splitting ones, and get that there exists a unique exact substructure such that it is a cluster quotient.
\end{abstract}
\section{\bf Introduction}
Let $\mathcal{T}$ be a triangulated category, $\bigtriangleup (\mathcal{T})$ the additive category consisting of the triangles. Neeman in his monograph \cite[Chapter 5.1]{N1} proved that the quotient of $\bigtriangleup (\mathcal{T})$ by modulo a homotopy-like relation must be an abelian category. The same is also true in the case of exact categories when we replace triangles with conflations \cite[Section 1]{N2}. In fact, the later
abelian structure was discussed by Lin in more detail \cite{L}. In homological algebra, it is known that an important way to produce abelian quotients is the theory of cluster tilting.

We roughly recall the cluster tilting theory from the above point. The notion of a cluster tilting object in a cluster category first appears in \cite{BMRRT}. The quotient of a cluster category by an additive closure of a cluster tilting object is isomorphic to the module category of the corresponding cluster tilted algebra, hence an abelian category, this is one of the main results in \cite{BMR}. Koenig and Zhu proved that this abelian quotient is available in any triangulated category with a cluster tilting subcategory \cite{KZ}. Under some mild conditions, Beligiannis proved that the abelian quotient can characterize cluster tilting subcategories \cite[Theorem F]{Be}. The cases of exact categories were also considered in \cite{DL,LZ}.

Let $(\mathcal{M},\mathcal{E})$ be an exact category with the exact structure $\mathcal{E}$. Then the category of conflations, denoted by $\mathcal{E(M)}$, is also an exact category with the exact structure computed degree-wise. We denote this exact structure over $\mathcal{M}$ by $(\mathcal{E(M)},\mathcal{E})$. Let $\mathcal{S}(\mathcal{M})$ be the class of splitting conflations. Our focus is that how to connect the abelian quotient $\mathcal{E}(\mathcal{M})/\mathcal{S}(\mathcal{M})$ with cluster tilting theory in $(\mathcal{E(M)},\mathcal{E})$.

Let us recall the definition of a cluster tilting subcategory in an exact category. An additive subcategory $\mathcal{P}$ of $(\mathcal{M},\mathcal{E})$ is a cluster tilting subcategory if (1) $\mathcal{P}$ is strongly functorially finite; (2) $\mathcal{P}$ is maximum self-orthogonal. It can be seen that this definition is equivalent to: ($1'$) for any $X\in \mathcal{M}$,
there are two conflations $0\to X \overset{\alpha }\to Q_0\to Q_1\to 0$ and $0\to P_1 \to P_0 \overset{\beta }\to X\to 0$ with $Q_0,\ Q_1,\ P_0$, and  $P_1\in \mathcal{P}$, and $\alpha $ (resp. $\beta $) is a $\mathcal{P}$-preenvelope (resp. $\mathcal{P}$-precover); ($2'$) $\mathcal{P}$ is self-orthogonal. We call $\mathcal{P}$ a $pseudo$-$cluster\ tilting\ subcategory$ if it satisfies ($1'$). An interesting phenomenon is that $\mathcal{S}(\mathcal{M})$ is always a pseudo-cluster tilting subcategory of $(\mathcal{E(M)},\mathcal{E})$, see Proposition \ref{S(M) pre}. For this reason, we consider the quotient of an exact category by a pseudo-cluster tilting subcategory, and get the following. Compare with \cite[Theorem F]{Be} and \cite[Theorem 1.6]{LZ}.

\begin{thm}\label{theorem 1.1}
Let $\mathcal{M}$ be an exact category and $\mathcal{P}$ a pseudo-cluster tilting subcategory.
\begin{enumerate}
\item[(1)] $\mathcal{M}$/$\mathcal{P}$ is a semi-abelian category;
\item[(2)] $\mathcal{M}$/$\mathcal{P}$ is an abelian category if and only if one of the following holds
\begin{enumerate}
\item[(i)] The subcategory $\mathcal{P}$ is self-orthogonal with respect to $\mathcal{S}$;
\item[(ii)] The subcategory $\mathcal{P}$ is self orthogonal with respect to $\mathcal{T}$.
\end{enumerate}
where $\mathcal{S}$ and $\mathcal{T}$ are two special classes of conflations, see Section \ref{section 3} for the detail.
\end{enumerate}
\end{thm}

Then we treat the pair $(\mathcal{S}(\mathcal{M}), (\mathcal{E(M)},\mathcal{E}))$, and extend the self-orthogonal conditions in Theorem \ref{theorem 1.1} to the exact substructure $(\mathcal{E}(\mathcal{M}),\mathcal{E}_0)$, conflations splitting in the middle degree. The relation between the abelian quotient $\mathcal{E}(\mathcal{M})/\mathcal{S}(\mathcal{M})$ and cluster tilting theory is the following. Compare with \cite[Section 1]{N2} and \cite[Theorem 4.8]{L}.

\begin{thm}\label{theorem 1.2}
Let $(\mathcal{E},\mathcal{M})$ be an exact category. Then $\mathcal{E}(\mathcal{M})/\mathcal{S}(\mathcal{M})$ is an abelian category. Moreover, $(\mathcal{E}(\mathcal{M}),\mathcal{E}_0)$ is the unique exact structure such that it is a cluster quotient.
\end{thm}

This paper is organized as follows. In section \ref{section 2}, we make the preliminary definitions and results. In section \ref{section 3}, we are denoted to proving Theorem \ref{theorem 1.1}, which is divided into two parts. In Section \ref{section 4}, we discuss the pair $(\mathcal{S}(\mathcal{M}),(\mathcal{E(M)},\mathcal{E}))$ and prove Theorem \ref{theorem 1.2}.
Throughout this paper, for an additive category, we always assume that it is closed under isomorphisms. For an additive full subcategory, we mean that it is closed under finite sums and summands.

\section{Preliminaries}\label{section 2}
   In this section, we introduce some definitions and homological facts for the later use. Let $\mathcal{M}$ be an additive category and $f:X\rightarrow Y$ a morphism. A weak kernel of $f$ is a morphism $g: K\to X$ with $fg=0$ such that any morphism $h$ with $fh=0$ factors through $g$. A weak kernel $g$ of $f$ is the kernel if and only if it is a monomorphism. Dually, there are corresponding facts for weak cokernels and cokernels. We write $k_f: \mathrm{Ker}f\rightarrow X$ for the kernel of $f:X\rightarrow Y$ and $c_f: Y\rightarrow \mathrm{Coker}f$ for its cokernel. The coimage of $f$, i.e., the cokernel of
   $k_f$ is denoted by $c_{k_{f}}:X\rightarrow \mathrm{Coim}f$. The image of $f$, i.e., the kernel of
   $c_f$ is denoted by $k_{c_{f}}:\mathrm{Im}f\rightarrow Y$. There is a unique morphism $\hat{f}:\mathrm{Coim}f\to \mathrm{Im}f$ making the following diagram commutative.

   $$\xymatrix{
		\mathrm{Ker}f\ar[rr]^{k_f}&&X\ar[rr]^{f}\ar[d]_{c_{k_{f}}}&&Y\ar[rr]^{c_f}&&\mathrm{Coker}f\\
		&&\mathrm{Coim}f\ar[rr]^{\hat f}&&\mathrm{Im}f\ar[u]_{k_{c_{f}}}\\
	}$$\

\begin{df}\label{Definition 2.1}
(\cite{Bu,R}) An additive category $\mathcal{M}$ with kernels and cokernels is called semi-abelian if for any morphism $f:X\rightarrow Y$, the canonical morphism $\hat f:\mathrm{Coim}f\rightarrow \mathrm{Im}f$ is both an epimorphism and a monomorphism; it is abelian if $\hat{f}$ is an isomorphism.
\end{df}

\begin{rmk}\label{rmk 2.2}
In some literatures a morphism which is both an epimorphism and a monomorphism is called regular \cite{R}. It can be seen that a semi-abelian category is abelian if and only if regular morphisms and isomorphisms coincide.
\end{rmk}

Let $\mathcal{P}$ be an additive full subcategory of $\mathcal{M}$. Let $\Hom_{\mathcal{P}}(X,Y)$ be the subset of morphisms in $\Hom_{\mathcal{M}}(X,Y)$, which factors through an object in $\mathcal{P}$. Then $\mathcal{I}:=\{f\in \Hom_{\mathcal{P}}(X,Y)|X,Y\in \mathcal{M}\}$ is an ideal of $\mathcal{M}$, the corresponding quotient category is denoted by $\mathcal{M}$/$\mathcal{P}$. This is an additive category with products computed in $\mathcal{M}$. There is an additive functor $\mathcal{Q}:\mathcal{M}\rightarrow \mathcal{M}/\mathcal{P}$ which is full and dense. For an object $X\in \mathcal{M}$, we denote by $\underline{X}$ the image of $X$ under $\mathcal{Q}$. Hence $\underline{X}$ is 0 in $\mathcal{M}/\mathcal{P}$ if and only if $X\in \mathcal{P}$. For a morphism $f:X\rightarrow Y$ in $\mathcal{M}$, we denote by $\underline{f}$ the imgae of $f$ under $\mathcal{Q}$. Then $\underline{f}$ is 0 in $\mathcal{M}/\mathcal{P}$ if and only if it factors through an object in $\mathcal{P}$. We need the following fact, see \cite[Theorem 2.2]{H}.

 \begin{lem}\label{iso}
Let $\mathcal{M}$ be an additive category and $\mathcal{P}$ an additive full subcategory of $\mathcal{M}$. Let $f:X\to Y$ be a morphism in $\mathcal{M}$. Then $\underline{f}$ is an isomorphism in $\mathcal{M}$/$\mathcal{P}$ if and only if there exist $P$ and $Q$ in $\mathcal{P}$ completing the following commutative diagram in $\mathcal{M}$,

$$\xymatrix{
  X \ar[d]_{i} \ar[r]^{f}
               & Y    \\
  X\oplus P \ar[r]^{\widetilde{f}}
                & Y\oplus Q  \ar[u]_{p}
          }
$$
where $\widetilde{f}$ is an isomorphism and the vertical morphisms are the canonical injection and projection respectively.
\end{lem}

\begin{rmk}\label{rmk 2.4}
In the above commutative diagram, we can take $i=\left({\begin{array}{*{10}{c}}
	{1}\\
    {0}\\
	\end{array}} \right)$, $p=\left( {\begin{array}{*{10}{c}}
	{{0}}&{{1}}\\
	\end{array}} \right)$. Hence we have $\widetilde{f}=\left({\begin{array}{*{10}{c}}
	{{f}}&{{b}}\\
	{{c}}&{{d}}\\
	\end{array}} \right)$ by the matrix notation with $b:P\to Y,~d:P\to Q$ and $c:X\to Q.$
\end{rmk}

We recall the definitions of precovers and preenvelopes respectively. A morphism $f: P\rightarrow X$ in $\mathcal{M}$ with $P\in \mathcal{P}$ is called a $\mathcal{P}$-precover of $X$ if any morphism $g: Q\rightarrow X$ with $Q\in \mathcal{P}$ factors through $f$. If any $X\in \mathcal{M}$ admits a $\mathcal{P}$-precover then $\mathcal{P}$ is called a contravariantly finite subcategory. The notions of preenvelopes and covariantly finite subcategories are defined dually.

Now suppose that $(\mathcal{M},\mathcal{E})$ is an exact category, where $\mathcal{E}$ is the corresponding exact structure, i.e., the class of conflations. If for any $X\in \mathcal{M}$, there exists a $\mathcal{P}$-precover $f: P\rightarrow X$ which is compatible with $\mathcal{E}$, that is, $f: P\rightarrow X$ is an inflation, then $\mathcal{P}$ is called strongly contravariantly finite. Dually, we have the notion of strongly covariantly finiteness. The subcategory $\mathcal{P}$ is called strongly functorially finite if $\mathcal{P}$ is both strongly contravariantly finite and strongly covariantly finite.

We say that a conflation $0\rightarrow X\rightarrow Y\rightarrow Z\rightarrow 0$ in $\mathcal{M}$ is $\Hom_{\mathcal{M}}(\mathcal{P},-)$-exact if for any $P\in \mathcal{P}$, the sequence $0\rightarrow \Hom_{\mathcal{M}}(P,X)\rightarrow \Hom_{\mathcal{M}}(P,Y)\rightarrow \Hom_{\mathcal{M}}(P,Z)\rightarrow 0$ is exact in the category of abelian groups. Dually, We say that a conflation $0\rightarrow X\rightarrow Y\rightarrow Z\rightarrow 0$ in $\mathcal{M}$ is $\Hom_{\mathcal{M}}(-,\mathcal{P})$-exact if for any $P\in \mathcal{P}$, the sequence  $0\rightarrow \Hom_{\mathcal{M}}(Z,P)\rightarrow \Hom_{\mathcal{M}}(Y,P)\rightarrow \Hom_{\mathcal{M}}(X,P)\rightarrow 0$ is exact.	

\begin{lem}\label{Lemma 2.5}
Let $\mathcal{M}$ be an exact category and $\mathcal{P}$ a full subcatrgory of $\mathcal{M}$.
\begin{enumerate}
\item[(1)] If $\mathcal{P}$ is strongly covariantly finte, then for any morphism $f:X\rightarrow Y$ in $\mathcal{M}$, there exist some $Q\in\mathcal{P}$ and a conflation $0\to X\rightarrow Y\oplus Q\to Z\to 0$ which is $\Hom_{\mathcal{M}}(-,\mathcal{P})$-exact.
\item[(2)] If $\mathcal{P}$ is strongly contravariantly finte, then for any morphism $g:Y\rightarrow Z$ in $\mathcal{M}$, there exist some $P\in\mathcal{P}$ and a conflation $0\to X\rightarrow Y\oplus P\to Z\to 0$ which is $\Hom_{\mathcal{M}}(\mathcal{P},-)$-exact.
\end{enumerate}
\end{lem}
\begin{proof}
(1) Since $\mathcal{P}$ is strongly covariantly finte, we can pick a $\mathcal{P}$-preenvelope $\alpha: X\rightarrow Q $ which is an inflation. Consider the following push-out of $\alpha:X\rightarrow Q$ and $f:X\rightarrow Y$.

$$\xymatrix{
	X\ar[r]^f\ar[d]_\alpha&Y\ar[d]\\
	Q\ar[r]&Z
}$$
We have a confation $0\rightarrow X\stackrel{\begin{pmatrix}f\\-\alpha\end{pmatrix}} {\longrightarrow} Y\oplus Q\stackrel{} {\longrightarrow} Z\rightarrow 0$ in $\mathcal{M}$ by \cite[Proposition 2.12]{Bu}. Since $\alpha$ is a $\mathcal{P}$-preenvolope, for any morphism $\beta:X\rightarrow Q^{\prime}$ with $Q^{\prime}\in \mathcal{P}$, there is $\gamma:Q\rightarrow Q^{\prime}$ such that $\beta=\gamma\alpha$, then we obtain the following commutative diagram in $\mathcal{M}$.
$$\xymatrix{
	&&Q^{\prime}\\
	0\ar[rr]&&X\ar[rr]_-{\begin{pmatrix}f\\-\alpha\end{pmatrix}}\ar[u]^\beta&&Y\oplus{Q}\ar[ull]_-{\begin{pmatrix}0&-\gamma\end{pmatrix}}\\
}$$
That is, the conflation $0\rightarrow X\stackrel{\begin{pmatrix}f\\-\alpha\end{pmatrix}} {\longrightarrow} Y\oplus Q\stackrel{} {\longrightarrow} Z\rightarrow 0$ is $\Hom_{\mathcal{M}}(-,\mathcal{P})$-exact.

(2) is the dual.		
\end{proof}
\section{The quotient category $\mathcal{M}/\mathcal{P}$}\label{section 3}
Let $\mathcal{M}$ be an exact category and $\mathcal{P}$ an additive full subcategory of $\mathcal{M}$. In this section, we consider the abelian structure of the quotient category $\mathcal{M}$/$\mathcal{P}$. We require the following setup.

\begin{enumerate}
\item[$\blacktriangle$] $\mathcal{P}$ is strongly covariantly finte, and for any $X\in \mathcal{M}$, there is a conflation $0\to X \overset{\alpha }\to Q_0\to Q_1\to 0$ with $Q_0, Q_1\in \mathcal{P}$ such that $\alpha $ is a $\mathcal{P}$-preenvelope.
\item[$\blacktriangledown$] $\mathcal{P}$ is strongly contravariantly finte, and for any $X\in \mathcal{M}$, there is a conflation $0\to P_1 \to P_0 \overset{\beta }\to X\to 0$ with $P_0, P_1\in \mathcal{P}$ such that $\beta $ is a $\mathcal{P}$-precover.
\end{enumerate}

\begin{rmk}
\begin{enumerate}
\item[(1)] The conflation $0\to X \overset{\alpha }\to Q_0\to Q_1\to 0$ in $\blacktriangle$ is $\Hom_{\mathcal{M}}(-,\mathcal{P})$-exact;
\item[(2)] The conflation $0\to P_1 \to P_0 \overset{\beta }\to X\to 0$ in $\blacktriangledown$ is $\Hom_{\mathcal{M}}(\mathcal{P},-)$-exact;
\item[(3)] It can be seen that $\mathcal{P}$ is a pseudo-cluster tilting subcategory if it satisfies both $\blacktriangle$ and $\blacktriangledown$.
\end{enumerate}
\end{rmk}

\begin{prop}\label{ker-cok exist}
Let $\mathcal{M}$ be an exact category and $\mathcal{P}$ a full subcategory of $\mathcal{M}$.

\begin{enumerate}
\item[(1)] If $\mathcal{P}$ satisfies condition $\blacktriangle$, then $\mathcal{M}$/$\mathcal{P}$ has cokernels.
\item[(2)] If $\mathcal{P}$ satisfies condition $\blacktriangledown$, then $\mathcal{M}$/$\mathcal{P}$ has kernels.
\end{enumerate}
\end{prop}

\begin{proof}
We only prove (2), and (1) is just the dual. Let $f:X\rightarrow Y$ be an arbitrary morphism in $\mathcal{M}$. By assumption, there is a conflation $0\rightarrow P_1\rightarrow P_0\overset{\beta } \to Y\to 0$ with $P_0, P_1$ $\in$ $\mathcal{P}$ and $\beta :P_0\rightarrow Y$ a $\mathcal{P}$-precover. Consider the following pull-back of $f$ and $\beta $ in $\mathcal{M}$.

$$\xymatrix{
	0\ar[rr]&&P_1\ar@{=} [d] \ar[rr]&&K\ar[d] \ar[rr]^k&&X\ar[rr] \ar[d]^f&&0\\
	0\ar[rr]&&P_1\ar[rr]&&P_0\ar[rr]^\beta &&Y\ar[rr]&&0	
}$$
We claim that \underline{$K$} is the kernel of \underline{$f$}. Assume that there is a morphism $g:U\rightarrow X$ in $\mathcal{M}$ such that \underline{$f$} \underline{$g$}$=0$ in $\mathcal{M}$/$\mathcal{P}$, then there exists an object $P$ in $\mathcal{P}$ such that $fg$ factors through $P$. Since $\beta :P_0\rightarrow Y$ is $\mathcal{P}$-precover, thus $f$$g$ factors through $P_0$. By the universal property of pull-back, we get that $g$ factors through $K$, hence \underline{$K$} is a weak kernel of \underline{$f$}.

It remains to show that \underline{$k$} is a monomorphism in $\mathcal{M}$/$\mathcal{P}$. Let $h:V\rightarrow K$ be in $\mathcal{M}$, and assume that $\underline{kh}=0$. Then there exist an object $P^{\prime}$$\in$$\mathcal{P}$, and morphisms $i:V\rightarrow P^{\prime}$ and $j:P^{\prime}\rightarrow X$ such that $kh=ji$. Since $\beta :P_0\rightarrow Y$ is a $\mathcal{P}$-precover, there is a morphism $l:P^{\prime}\rightarrow P_0$ such that $fj=\beta l$. We have the following commutative diagram in $\mathcal{M}$.
$$\xymatrix{
 	&&&&V\ar[d]_h\ar@{.>}[rr]^i&&P^{\prime}\ar@{.>}[d]^j\ar@{.>}[dll]_n\ar@{.>}[ddll]^l\\
 	0\ar[rr]&&P_1\ar@{=} [d] \ar[rr]&&K\ar[d] \ar[rr]^k&&X\ar[rr] \ar[d]^f&&0\\
 	0\ar[rr]&&P_1\ar[rr]&&P_0\ar[rr]^\beta &&Y\ar[rr]&&0
 }$$
By the universal property of pull-back, there is a unique morphism $n:P^{\prime}\rightarrow K$ such that $j=kn$. Then we get that $k(ni-h)=0$. As a consequence, the morphism $ni-h$ factors through $P_1$, thus $\underline{ni-h}=0$. It implies that $\underline{h}=0$ and $\underline{k}$ is a monomorphism in $\mathcal{M}$/$\mathcal{P}$. Therefore, $\underline{k}: \underline{K}\to \underline{X}$ is the kernel of $\underline{f}$ in $\mathcal{M}$/$\mathcal{P}$.
\end{proof}

\begin{lem}\label{ker-coker}
Let $\mathcal{M}$ be an exact category and $\mathcal{P}$ a full subcategory of $\mathcal{M}$.

\begin{enumerate}
\item[(1)] If $\mathcal{P}$ satisfies condition $\blacktriangle$, then for any conflation $0\rightarrow X\stackrel{f} {\longrightarrow} Y\stackrel{g} {\longrightarrow} Z\rightarrow 0$ which is $\Hom_{\mathcal{M}}(-,\mathcal{P})$-exact, we have that $\underline{Z}$ is the cokernel of $\underline{f}$.
\item[(2)] If $\mathcal{P}$ satisfies condition $\blacktriangledown$, then for any conflation $0\rightarrow X\stackrel{f} {\longrightarrow} Y\stackrel{g} {\longrightarrow} Z\rightarrow 0$ which is $\Hom_{\mathcal{M}}(\mathcal{P},-)$-exact, we have that $\underline{X}$ is the kernel of $\underline{g}$.
\end{enumerate}
\end{lem}

\begin{proof}
For the convenience, we only prove (1), and (2) can be shown by the similar argument. For the object $X$ in $\mathcal{M}$, there exists a conflation $0\rightarrow X\overset{\beta } \to Q_0\stackrel{} {\longrightarrow} Q_1\rightarrow 0$, with $Q_0$, $Q_1\in\mathcal{P}$ and $\beta :X\to Q_0$ a $\mathcal{P}$-preenvelope. Consider the following push-out diagram of $f:X\rightarrow Y$ and $\beta :X\rightarrow Q_0$ in $\mathcal{M}$.

	$$\xymatrix{
		&&0\ar[d]&&0\ar[d]\\
		0\ar[rr]&&X\ar[d]_{\beta }\ar[rr]^f&&Y\ar[rr]^g\ar[d]\ar@{.>}[ddrr]^j&&Z\ar@{=}[d]\ar[rr]&&0\\
		0\ar[rr]&&Q_0\ar[d]\ar[drrrr]^{i\ \ \ \ }\ar[rr]^r&&C\ar[d]\ar@{.>}[drr]\ar[rr]&&Z\ar[rr]&&0\\
		&&Q_1\ar@{=}[rr]\ar[d]&&Q_1\ar[d]&&Q\\
		&&0&&0
	}$$
Assume that $Q$ is an arbitary object in $\mathcal{P}$, and $i:Q_0\rightarrow Q$ an arbitrary morphism. Since the conflation $0\rightarrow X\stackrel{f} {\longrightarrow} Y\stackrel{g} {\longrightarrow} Z\rightarrow 0$ is $\Hom_{\mathcal{M}}(-,\mathcal{P})$-exact, thus there is a morphism $j:Y\rightarrow Q$ such that $i\beta =jf$.
By the universal property of push-out, there is a unique morphism $C\rightarrow Q$ such that $i$ factors through $r:Q_0\rightarrow C$. This shows that the second row is $\Hom_{\mathcal{M}}(-,\mathcal{P})$-exact. Since $Q_0\in \mathcal{P}$, we infer that the second row splits.
Therefore we get that $C\cong Q_0\oplus{Z}$ in $\mathcal{M}$, and hence $\underline{j}:\underline{C} \to \underline{Z}$ is an isomorphism in $\mathcal{M}$/$\mathcal{P}$ by Lemma \ref{iso}. By Proposition \ref{ker-cok exist}, $\underline{g}:\underline{Y}\to \underline{Z}$ is the cokernel of $\underline{f}$.
\end{proof}

\begin{thm} \label{Th semi-abel}
Let $\mathcal{M}$ be an exact category and $\mathcal{P}$ a pseudo-cluster tilting subcategory of $\mathcal{M}$. Then $\mathcal{M}$/$\mathcal{P}$ is a semi-abelian category.
\end{thm}
\begin{proof}
Let $f:X\to Y$ be in $\mathcal{M}$. By Lemma \ref{Lemma 2.5}, we have a conflation  $0\rightarrow X\to  Y\oplus Q \rightarrow Z\rightarrow 0$ with $Q\in$ $\mathcal{P}$, which is $\Hom_{\mathcal{M}}(-,\mathcal{P})$-exact. By Lemma \ref{ker-coker}, we know that $\underline{Z}$ is the cokernel of $\underline{f}$. We can also get a conflation $0\rightarrow X^{\prime}\rightarrow Y\oplus Q \oplus P\rightarrow Z\rightarrow 0$ with $P \in \mathcal{P}$, which is $\Hom_{\mathcal{M}}(\mathcal{P},-)$-exact, then $\underline{X^{\prime}}$ is the kernel of $c_{\underline{f}}$. Let $i:Y\oplus{Q}\rightarrow Y\oplus{Q}\oplus{P}$ be the canonical injection. We obtain the following commutative diagram.

$$\xymatrix{
		&&0\ar[d]&&0\ar[d]\\
		0\ar[rr]&&X\ar[d]^g\ar[rr]\ar@{.>}[ddll]_h&&Y\oplus{Q}\ar[rr]^{c_{\underline{f}}}\ar@{.>}[ddllll]^j\ar[d]^i&&Z\ar@{=}[d]\ar[rr]&&0\\
		0\ar[rr]&&X^{\prime}\ar[d]\ar[rr]&&Y\oplus{Q}\oplus{P}\ar[d]\ar[rr]&&Z\ar[rr]&&0\\
		Q^{\prime}&&P\ar@{=}[rr]\ar[d]&&P\ar[d]\\
		&&0&&0&&(\mathrm{Fig.}\ \ 3.1)
}$$

We claim that $\underline{g}$ is an epimorphism in $\mathcal{M}$/$\mathcal{P}$. In fact, for any morphism $h:X\rightarrow Q^{\prime}$ with $Q^{\prime} \in \mathcal{P}$, since the first row is $\Hom_{\mathcal{M}}(-,\mathcal{P})$-exact, there is a morphism $j:Y\oplus{Q}\rightarrow Q^{\prime}$, such that $h$ factors through $X\rightarrow Y\oplus{Q}$. Since the middle column splits, $j$ factors through $i$. This means that $h$ factors through $g$, and hence the first column is $\Hom_{\mathcal{M}}(-,\mathcal{P})$-exact. By Lemma \ref{ker-coker}, $\underline{P}$ is the cokernel of $\underline{g}$ in $\mathcal{M}$/$\mathcal{P}$. We infer that $\underline{g}$ is an epimorphism in $\mathcal{M}$/$\mathcal{P}$ due to the fact that $P\in\mathcal{P}$.

Since $\mathcal{M}$/$\mathcal{P}$ has both kernels and cokernels, we then obtain the following commutative diagram in $\mathcal{M}$/$\mathcal{P}$.
\vspace{0.4cm}
$$\xymatrix{
		\mathrm{Ker}\underline{f}\ar[rr]^{k_{\underline{f}}}&&\underline{X}\ar[d]_{c_{k_{\underline{f}}}}\ar[rr]^{\underline{f}}\ar@{.>}[drr]^{\underline{g}}&&\underline{Y}\ar[rr]^{c_{\underline{f}}}&&\underline{Z}\\
		&&\mathrm{Coim}\underline{f}\ar[rr]^{\underline{\hat{f}}}&&\underline{X^{\prime}}\ar[u]_{k_{c_{\underline{f}}}}
}$$

\vspace{0.4cm}
It can be seen that the canonical morphism $\underline{\hat{f}}$ is epimorphism. Dually, we can show that $\underline{\hat{f}}$ is a monomorphism. Therefore, $\mathcal{M}$/$\mathcal{P}$ is a semi-abelian category as desired.
\end{proof}

In the rest of this section, we test the abelianness of the quotient $\mathcal{M}$/$\mathcal{P}$. Since we have proven that $\mathcal{M}$/$\mathcal{P}$ is always a semi-abelian category, we only need to show that any regular morphism in $\mathcal{M}$/$\mathcal{P}$ must be an isomorphism. Note that if $\mathcal{P}$ is a cluster tilting subcategory, i.e., $\mathcal{P}$ is self-orthogonal with respect to all conflations, then $\mathcal{M}$/$\mathcal{P}$ is always an abelian category by \cite{DL}. In the following, we introduce two special classes of conflations $\mathcal{S}$ and $\mathcal{T}$.

A conflation $0\rightarrow S_1\rightarrow S_2\rightarrow S_3\rightarrow 0$ belongs to $\mathcal{S}$, if it appears in the following commutative diagram, where the conflation in the second column is $\Hom_{\mathcal{M}}(\mathcal{P},-)$-exact and the one in the third row is $\Hom_{\mathcal{M}}(-,\mathcal{P})$-exact.

\vspace{0.5cm}
$$\xymatrix{
 	&&0\ar[d]&&0\ar[d]\\
 	&&U\ar[d]\ar@{=}[rr]&&U\ar[d]\\
 	0\ar[rr]&&S_1\ar[d]\ar[rr]&&S_2\ar[d]\ar[rr]&&S_3\ar@{=}[d]\ar[rr]&&0\\
 	0\ar[rr]&&M_1\ar[d]\ar[rr]&&M_2\ar[d]\ar[rr]&&S_3\ar[rr]&&0\\
 	&&0&&0
 }$$

\vspace{0.5cm}
Dually, a conflation $0\rightarrow T_1\rightarrow T_2\rightarrow T_3\rightarrow 0$ belongs to $\mathcal{T}$, if it appears in the following commutative diagram, where the conflation in the second column is $\Hom_{\mathcal{M}}(-,\mathcal{P})$-exact, and the one in the first row is $\Hom_{\mathcal{M}}(\mathcal{P},-)$-exact.

$$\xymatrix{
	&&&&0\ar[d]&&0\ar[d]\\
	0\ar[rr]&&T_1\ar@{=}[d]\ar[rr]&&N_2\ar[d]\ar[rr]&&N_3\ar[d]\ar[rr]&&0\\
	0\ar[rr]&&T_1\ar[rr]&&T_2\ar[d]\ar[rr]&&T_3\ar[d]\ar[rr]&&0\\
	&&&&V\ar@{=}[rr]\ar[d]&&V\ar[d]\\
	&&&&0&&0
}$$

\begin{rmk} \label{rmk S and T}The readers can check that the following two assertions hold.
\begin{enumerate}
\item[(1)] A conflation which is either $\Hom_{\mathcal{M}}(\mathcal{P},-)$-exact or $\Hom_{\mathcal{M}}(-,\mathcal{P})$-exact belongs to $\mathcal{S}$;
\item[(2)] A conflation which is either $\Hom_{\mathcal{M}}(\mathcal{P},-)$-exact or $\Hom_{\mathcal{M}}(-,\mathcal{P})$-exact belongs to $\mathcal{T}$.
\end{enumerate}
\end{rmk}

\begin{thm} \label{main Th}
Let $\mathcal{M}$ be an exact category and $\mathcal{P}$ a pseudo-cluster tilting subcategory of $\mathcal{M}$. Then the following are equivalent.

\begin{enumerate}
\item[(1)] $\mathcal{M}$/$\mathcal{P}$ is an abelian category;
\item[(2)] The subcategory $\mathcal{P}$ is self-orthogonal with respect to $\mathcal{S}$;
\item[(3)] The subcategory $\mathcal{P}$ is self orthogonal with respect to $\mathcal{T}$.
\end{enumerate}
\end{thm}

\begin{proof}
We only prove that (1) and (2) are equivalent, the equivalence of (1) and (3) is the dual.

(2)$\Rightarrow$(1) We will show that a regular morphism in $\mathcal{M}$/$\mathcal{P}$ must be an isomorphism. For convenience, we keep the notation from the proof of Theorem \ref{Th semi-abel}. Let $\underline{f}:\underline{X}\rightarrow \underline{Y}$ be a monomorphism in $\mathcal{M}$/$\mathcal{P}$ induced by $f:X \to Y$ in $\mathcal{M}$. We then obtain the conflation
$0 \to X\overset{g}\to X'\to P\to 0$ with $P\in \mathcal{P}$, which is $\Hom_{\mathcal{M}}(-,\mathcal{P})$-exact as in Fig. 3.1. For the object $X^{\prime}$, there exists a conflation $0\rightarrow P_1\rightarrow P_0\rightarrow X^{\prime}\rightarrow 0$ with $P_0$, $P_1 \in \mathcal{P}$, which is $\Hom_{\mathcal{M}}(\mathcal{P},-)$-exact. Consider the following commutative diagram.

$$\xymatrix{
 	&&0\ar[d]&&0\ar[d]\\
 	&&P_1\ar[d]\ar@{=}[rr]&&P_1\ar[d]\\
 	0\ar[rr]&&K\ar[d]\ar[rr]&&P_0\ar[d]\ar[rr]&&P\ar@{=}[d]\ar[rr]&&0\\
 	0\ar[rr]&&X\ar[d]\ar[rr]^{g}&&X'\ar[d]\ar[rr]&&P\ar[rr]&&0\\
 	&&0&&0
 }$$
It is shown that the conflation $0\rightarrow K\rightarrow P_0\rightarrow P\rightarrow0$ belongs to $\mathcal{S}$. Consider the following commutative diagram in $\mathcal{M}$ from Fig. 3.1.

$$\xymatrix{
	X\ar[r]^{f\ \ }\ar[d]_g &Y\oplus Q\ar[d]_{i}\\
	X'\ar[r]&Y\oplus Q\oplus P
}$$
It induces the following commutative diagram in the semi-abelian category $\mathcal{M}$/$\mathcal{P}$.
$$\xymatrix{
	\underline{X}\ar[r]^{\underline{f}\ \ }\ar[d]_{\underline{g}} &\underline{Y}\oplus \underline{Q}\ar[d]_{\underline{i}}\\
	\underline{X'}\ar[r]^{k_{c_{\underline{f}}}\ \ \ \ }&\underline{Y}\oplus \underline{Q}\oplus \underline{P}
}$$
Since $\underline{i}$ is an isomorphism by Lemma \ref{iso} and $\underline{f}$ is a monomorphism, we infer that $\underline{g}$ is a monomorphism.
This implies that $K$ belongs to $\mathcal{P}$, since $\underline{K}$ is the kernel of $\underline{g}$ by Proposition \ref{ker-cok exist}. We have that the conflation $0\rightarrow K\rightarrow P_0\rightarrow P\rightarrow 0$ splits. As a consequence, $0 \to X\overset{g}\to X'\to P\to 0$ also splits. Hence in the above diagram, the two columns are all isomorphisms in $\mathcal{M}$/$\mathcal{P}$, hence $\underline{f}$ is a kernel. This means any monomorphism in $\mathcal{M}$/$\mathcal{P}$ must be a kernel. Now the remains are well known, see \cite[page 199]{M}.

(1)$\Rightarrow$(2) Let $0\rightarrow P\rightarrow X\rightarrow Q\rightarrow 0$ be a conflation with $P,\ Q\in\mathcal{P}$ belongs to $\mathcal{S}$. Then we have the following commutative diagram
$$\xymatrix{
		&&0\ar[d]&&0\ar[d]\\
	&&U\ar[d]\ar@{=}[rr]&&U\ar[d]\\
	0\ar[rr]&&P\ar[d]\ar[rr]&&X\ar[d]^r\ar[rr]^p&&Q\ar@{=}[d]\ar[rr]&&0\\
	0\ar[rr]&&M_1\ar[d]\ar[rr]^m&&M_2\ar[d]\ar[rr]^q&&Q\ar[rr]&&0\\
	&&0&&0
}$$
where the third row is $\Hom_{\mathcal{M}}(-,\mathcal{P})$-exact. We get that $\underline{m}:\underline{M_1}\rightarrow \underline{M_2}$ is an epimorphism by Lemma \ref{ker-coker}. Since the second column is $\Hom_{\mathcal{M}}(\mathcal{P},-)$-exact, it can be checked that $P$ is a weak kernel of $\underline{m}$ in $\mathcal{M}$/$\mathcal{P}$, hence $\underline{m}$ is a monomorphism. By the assumption, $\mathcal{M}$/$\mathcal{P}$ is an abelian category, it shows that $\underline{m}$ is an isomorphism.

We have the following commutative diagram in $\mathcal{M}$ with $P',\ P''\in \mathcal{P}$ by lemma \ref{iso}.

$$\xymatrix{
  M_{1} \ar[d]_{i} \ar[r]^{m}
               & M_{2}    \\
  M_{1}\oplus P' \ar[r]^{\widetilde{m}}
                & M_{2}\oplus P''  \ar[u]_{p}
          }
$$
Since $\widetilde{m}=\left({\begin{array}{*{10}{c}}
	{{m}}&{{b}}\\
	{{c}}&{{d}}\\
	\end{array}} \right)$ is an isomorphism. Write its inverse $\widetilde{m}^{-1}=\left({\begin{array}{*{10}{c}}
	{{m_{1}}}&{{b_{1}}}\\
	{{c}_{1}}&{{d}_{1}}\\
	\end{array}} \right)$.
We have that $m_{1}m+b_{1}c=1$.
Consider the morphism $c:M_1\to P''$, since $0 \to M_1\overset{m}\to M_2\to Q\to 0$ is $\Hom_{\mathcal{M}}(-,\mathcal{P})$-exact, there exists
$v:M_2\to P''$ such that $c=vm$. Hence $m$ is a section and $0 \to M_1\overset{m}\to M_2\to Q\to 0$ splits.
It can be seen that $\Hom_{\mathcal{M}}(\mathcal{P},p)=\Hom_{\mathcal{M}}(\mathcal{P},r)\Hom_{\mathcal{M}}(\mathcal{P},q)$ is an epimorphism, this means $0\rightarrow P\rightarrow X\rightarrow Q\rightarrow 0$ splits. Hence the full subcategory $\mathcal{P}$ is self-orthogonal with respect to $\mathcal{S}$.
\end{proof}

\section{The category of conflations revisited}\label{section 4}
Let $(\mathcal{M},\mathcal{E})$ be an exact category with the exact structure $\mathcal{E}$. In this section, we revisit the category of conflations of $\mathcal{M}$, denoted by $\mathcal{E}(\mathcal{M})$. The objects in $\mathcal{E}(\mathcal{M})$ are all conflations of $\mathcal{M}$, we always view a conflation $0\rightarrow X_1\stackrel{x_1}{\longrightarrow}X_2\stackrel{x_2}{\longrightarrow} X_3\rightarrow 0$ as a complex concentrated in degree $-1$,~0 and 1, written as $X^{\bullet}$ for short. A morphism $f^{\bullet}:X^{\bullet}\rightarrow Y^{\bullet}$ is the following commutative diagram.
$$\xymatrix{
	0\ar[r]&X_1\ar[d]_{f_1}\ar[r]^{x_1}&X_2\ar[d]_{f_2}\ar[r]^{x_2}&X_3\ar[d]_{f_3}\ar[r]&0\\
	0\ar[r]&Y_1\ar[r]^{y_1}&Y_2\ar[r]^{y_2}&Y_3\ar[r]&0	
	}$$
It is well known that $\mathcal{E}(\mathcal{M})$ is an additive category. We introduce four exact structures over $\mathcal{E}(\mathcal{M})$. The first one is the usual exact structure computed degree-wise, written as $(\mathcal{E}(\mathcal{M}),\mathcal{E})$. The second one is $(\mathcal{E}(\mathcal{M}),\mathcal{E}_{0}^{-1})$, that is, conflations in $(\mathcal{E}(\mathcal{M}),\mathcal{E})$ splitting in degree 0 and $-1$. Dually, we have $(\mathcal{E}(\mathcal{M}),\mathcal{E}_{0}^{1})$. The last one is $(\mathcal{E}(\mathcal{M}),\mathcal{E}_{0})$, conflations in $(\mathcal{E}(\mathcal{M}),\mathcal{E})$ splitting in degree 0. Then it can be seen that

$$(\mathcal{E}(\mathcal{M}),\mathcal{E}_{0}^{-1})\preceq (\mathcal{E}(\mathcal{M}),\mathcal{E}_{0})\preceq (\mathcal{E}(\mathcal{M}),\mathcal{E})$$

$$(\mathcal{E}(\mathcal{M}),\mathcal{E}_{0}^{1})\preceq (\mathcal{E}(\mathcal{M}),\mathcal{E}_{0})\preceq (\mathcal{E}(\mathcal{M}),\mathcal{E})$$
In the following, we always write a conflation $0\rightarrow X_1\stackrel{x_1}{\longrightarrow}X_2\stackrel{x_2}{\longrightarrow} X_3\rightarrow 0$ in $\mathcal{M}$ as $X_1\stackrel{x_1}{\longrightarrow}X_2\stackrel{x_2}{\longrightarrow} X_3$ for simplicity.

Let $\mathcal{S}(\mathcal{M})$ be the full subcategory of $\mathcal{E}(\mathcal{M})$ consisting of all splitting conflations.
For any object $X^{\bullet}:0\rightarrow X_1\stackrel{x_1} {\longrightarrow} X_2\stackrel{x_2} {\longrightarrow} X_3\rightarrow 0$ in $\mathcal{E}(\mathcal{M})$, we obtain the following diagram, where $i$ and $p$ are the canonical injection and projection respectively.
$$\xymatrix{
0\ar[rr]&&0\ar[rr]\ar[d]&&X_1\ar[d]^-{i}\ar@{=}[rr]&&X_1\ar[rr]\ar[d]^{x_1}&&0\\
0\ar[rr]&&X_1\ar@{=}[d]\ar[rr]^-{\begin{pmatrix}1\\-x_1\end{pmatrix}}&&X_1\oplus{X_2}\ar[d]^-{p}\ar[rr]^-{\begin{pmatrix}x_1&1\end{pmatrix}}&&X_2\ar[d]^{x_2}\ar[rr]&&0\\
0\ar[rr]&&X_1\ar[rr]^{-x_1}&&X_2\ar[rr]^{x_2}&&X_3\ar[rr]&&0\\
}$$

It can be checked that the first and the second columns in the above diagram are splitting, denoted by $P^{\bullet}_1$ and $P^{\bullet}_0$ respectively. Thus we obtain a conflation $0\rightarrow P^{\bullet}_1\rightarrow P^{\bullet}_0 \overset{\alpha  ^\bullet}\to X^{\bullet}\to 0$ in $(\mathcal{E}(\mathcal{M}),\mathcal{E}_{0}^{-1})\preceq (\mathcal{E}(\mathcal{M}),\mathcal{E})$ with $P^{\bullet}_0,\ P^{\bullet}_1\in \mathcal{S}(\mathcal{M})$. Dually, we get a conflation $0\rightarrow X^{\bullet} \overset{\beta   ^\bullet}\to Q^{\bullet}_0 \rightarrow Q^{\bullet}_1\rightarrow 0$ in $(\mathcal{E}(\mathcal{M}),\mathcal{E}_{0}^{1})\preceq(\mathcal{E}(\mathcal{M}),\mathcal{E})$ with $Q^{\bullet}_0,\ Q^{\bullet}_1\in \mathcal{S}(\mathcal{M})$.

Now let $Y^{\bullet}$ be an arbitrary object in $\mathcal{S}(\mathcal{M})$. Then up to isomorphisms, we may assume that $Y^{\bullet}$ is of the form $0\rightarrow Y_1\stackrel{i} \to Y_1\oplus{Y_2}\stackrel{p} \to Y_2\rightarrow 0$ with $i$ and $p$ the canonical injection and projection respectively. Let $g^{\bullet}:Y^{\bullet}\rightarrow X^{\bullet}$ be arbitrary. Then we have the following diagram.

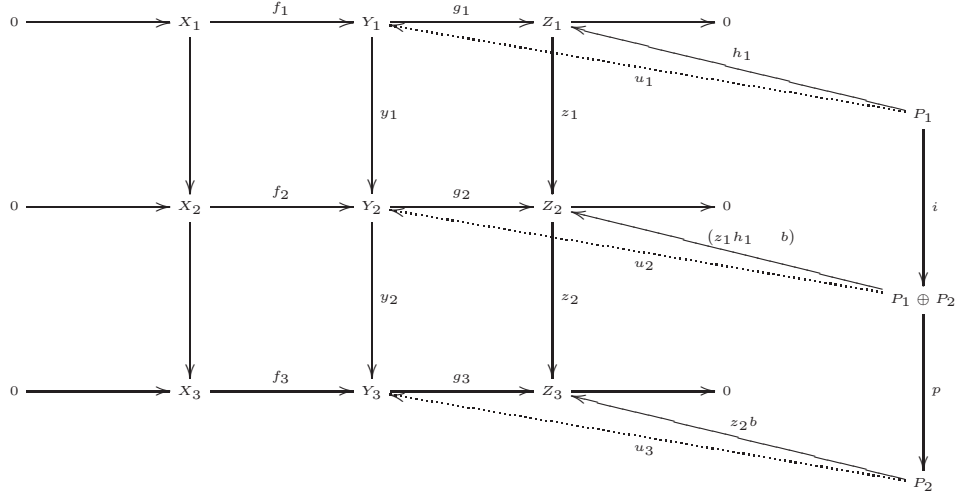
$$\xymatrix{
&&&&Y_1\ar[dd]^-{i}\ar[dll]^{g_1}\ar@{.>}[dllll]_{g_1}\\
X_1\ar@{=}[rr]\ar[dd]_-{i}&&X_1\ar[dd]\\
&&&&Y_1\oplus{Y_2}\ar[dd]^-{p}\ar[dll]^-{\begin{pmatrix}x_1 g_1&g'\end{pmatrix}}\ar@{.>}[dllll]_-{\begin{pmatrix}g_1&0\\0&g'\end{pmatrix}}\\
X_1\oplus{X_2}\ar[dd]_-{p}\ar[rr]_-{\begin{pmatrix}x_1&1\end{pmatrix}}&&X_2\ar[dd]\\
&&&&Y_2\ar[dll]^{x_2 g'}\ar@{.>}[dllll]_{g'}\\
X_2\ar[rr]_{x_2}&&X_3\\
}$$
This shows that $P^{\bullet}_0\overset{\alpha ^\bullet}\rightarrow X^{\bullet}$ is a $\mathcal{S}(\mathcal{M})$-precover. Dually, we can show that $X^{\bullet} \overset{\beta  ^\bullet}\rightarrow Q^{\bullet}_0$ is a $\mathcal{S}(\mathcal{M})$-preenvelope. Therefore, $\mathcal{S}(\mathcal{M})$ is a strongly functorially finite subcategory of $(\mathcal{E}(\mathcal{M}),\mathcal{E})$. We have already proven the following.

\begin{prop}\label{S(M) pre}
Let $(\mathcal{M},\mathcal{E})$ be an exact category. Then $\mathcal{S}(\mathcal{M})$ is a pseudo-cluster tilting subcategory in $(\mathcal{E(M)},\mathcal{E})$.
\end{prop}

\begin{lem}\label{char E(M)}
Let	$(\mathcal{E}(\mathcal{M}),\mathcal{E})$ and $\mathcal{S}(\mathcal{M})$ be as above. Let $S:0\rightarrow X^{\bullet}\stackrel{f^{\bullet}}{\longrightarrow} Y^{\bullet}\stackrel{g^{\bullet}}{\longrightarrow}Z^{\bullet}\rightarrow 0$ be a conflation in $(\mathcal{E}(\mathcal{M}),\mathcal{E})$.

\begin{enumerate}
\item[(1)] $S$ is $\Hom_{\mathcal{E}(\mathcal{M})}(\mathcal{S}(\mathcal{M}),-)$-exact if and only if it belongs to $(\mathcal{E}(\mathcal{M}),\mathcal{E}_{0}^{-1})$;
\item[(2)] $S$ is $\Hom_{\mathcal{E}(\mathcal{M})}(-,\mathcal{S}(\mathcal{M}))$-exact if and only if it belongs to $(\mathcal{E}(\mathcal{M}),\mathcal{E}_{0}^{1})$.
\end{enumerate}
\end{lem}
\begin{proof}We only prove (1), and (2) is the dual. Suppose that $S$ is $\Hom_{\mathcal{E}(\mathcal{M})}(\mathcal{S}(\mathcal{M}),-)$-exact. We take a conflation $0\rightarrow P^{\bullet}_1\rightarrow P^{\bullet}_0 \overset{r^\bullet}\to Z^{\bullet}\to 0$ in $(\mathcal{E}(\mathcal{M}),\mathcal{E}_{0}^{-1})$ with $P^{\bullet}_0,P^{\bullet}_1\in \mathcal{S}(\mathcal{M})$ by Proposition \ref{S(M) pre}. Then $r^{\bullet}$ factors through $g^{\bullet}$. This means $S$ splits in degree $-1$ and 0, as a consequence, it belongs to $(\mathcal{E}(\mathcal{M}),\mathcal{E}_{0}^{-1})$.

Conversely, assume that $S$ belongs to $(\mathcal{E}(\mathcal{M}),\mathcal{E}_{0}^{-1})$. Let $P^{\bullet}$ be an arbitrary object in $\mathcal{S}(\mathcal{M})$. Then up to isomorphisms, we may assume that $P^{\bullet}$ is of the form $0\rightarrow P_1\stackrel{i} \to P_1\oplus{P_2}\stackrel{p} \to P_2\rightarrow 0$ with $i$ and $p$ the canonical injection and projection respectively. Let $h^{\bullet}:P^{\bullet}\rightarrow Z^{\bullet}$ be arbitrary. Then we have the following diagram.

{\tiny $$\xymatrix{
	0\ar[rr]&&X_1\ar[rr]^{f_1}\ar[dd]&&Y_1\ar[dd]^{y_1}\ar[rr]^{g_1}&&Z_1\ar[rr]\ar[dd]^{z_1}&&0\\
	&&&&&&&&&&P_1\ar[dd]^{i}\ar@{.>}[ullllll]^{u_1}\ar[ullll]_{h_1}\\
	0\ar[rr]&&X_2\ar[dd]\ar[rr]^{f_2}&&Y_2\ar[dd]^{y_2}\ar[rr]^{g_2}&&Z_2\ar[dd]^{z_2}\ar[rr]&&0\\ &&&&&&&&&&P_1\oplus{P_2}\ar[dd]^{p}\ar[ullll]_{\begin{pmatrix}z_{1}h_{1}&b\end{pmatrix}}\ar@{.>}[ullllll]^{u_2}\\
	0\ar[rr]&&X_3\ar[rr]^{f_3}&&Y_3\ar[rr]^{g_3}&&Z_3\ar[rr]&&0\\
	&&&&&&&&&&P_2\ar[ullll]_{z_2b}\ar@{.>}[ullllll]^{u_3}\\	
}$$}

There are two morphisms $s_1:Z_1\rightarrow Y_1$ and $s_2: Z_2\rightarrow Y_2$, such that $g_1 s_1=1$ and $g_2 s_2=1$, respectively. Let $u_1 = s_1 h_1$, $u_2=\begin{pmatrix}y_1 s_1 h_1&s_2b\end{pmatrix}$ and $u_3 = y_2s_2b$. It is easy to check that $u^{\bullet}:P^{\bullet}\rightarrow Y^{\bullet}$ is a morphism such that $h^{\bullet}$ factors through $g^{\bullet}$. Therefore, the conflation $S$ is $\Hom_{\mathcal{E}(\mathcal{M})}(\mathcal{S}(\mathcal{M}),-)$-exact.
\end{proof}

\begin{lem}\label{min con s}
Let	$\mathcal{S}$ and $\mathcal{T}$ be as in Section \ref{section 3}. Then we have
\begin{enumerate}
\item[(1)] $(\mathcal{E}(\mathcal{M}),\mathcal{E}_0)$ is the minimum exact substructure in $(\mathcal{E}(\mathcal{M}),\mathcal{E})$ that contains $\mathcal{S}$;
\item[(2)] $(\mathcal{E}(\mathcal{M}),\mathcal{E}_0)$ is the minimum exact substructure in $(\mathcal{E}(\mathcal{M}),\mathcal{E})$ that contains $\mathcal{T}$.
\end{enumerate}
\end{lem}
\begin{proof}We only prove (1), and (2) is the dual. It can be easily checked that $\mathcal{S}\subseteq (\mathcal{E}(\mathcal{M}),\mathcal{E}_0)$. In fact, for a given conflation $0\rightarrow S_1^\bullet\rightarrow S_2^\bullet \rightarrow S_3^\bullet \rightarrow 0 \in \mathcal{S}$, there is the following commutative diagram, where the second column is $\Hom_{\mathcal{E}(\mathcal{M})}(\mathcal{S}(\mathcal{M}),-)$-exact and the third row is $\Hom_{\mathcal{E}(\mathcal{M})}(-,\mathcal{S}(\mathcal{M}))$-exact.

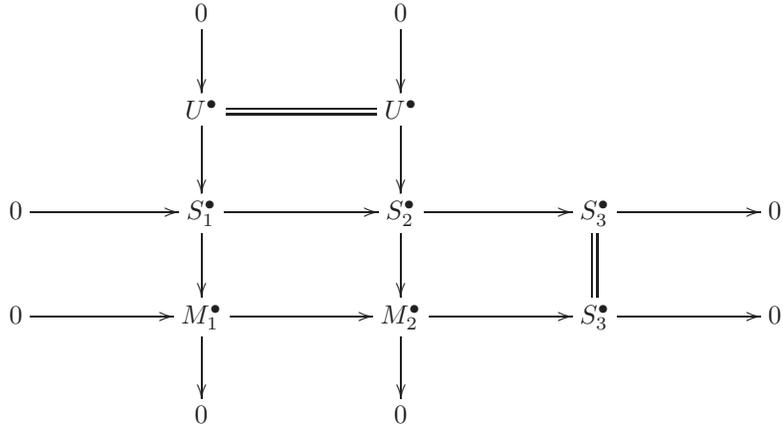
$$\xymatrix{
 	&&0\ar[d]&&0\ar[d]\\
 	&&U^\bullet\ar[d]\ar@{=}[rr]&&U^\bullet\ar[d]\\
 	0\ar[rr]&&S_1^\bullet\ar[d]\ar[rr]&&S_2^\bullet\ar[d]\ar[rr]&&S_3^\bullet\ar@{=}[d]\ar[rr]&&0\\
 	0\ar[rr]&&M_1^\bullet\ar[d]\ar[rr]&&M_2^\bullet\ar[d]\ar[rr]&&S_3^\bullet\ar[rr]&&0\\
 	&&0&&0
 }$$
It implies that the second column is in $(\mathcal{E}(\mathcal{M}),\mathcal{E}_{0}^{-1})$ and the third row is in $(\mathcal{E}(\mathcal{M}),\mathcal{E}_{0}^{1})$ by Lemma \ref{char E(M)}. Then it is trivial that $0\rightarrow S_1^\bullet\rightarrow S_2^\bullet\rightarrow S_3^\bullet\rightarrow 0 \in \mathcal{S}$ is in $(\mathcal{E}(\mathcal{M}),\mathcal{E}_{0})$.

Now suppose that $(\mathcal{E}(\mathcal{M}),\mathcal{E}')\preceq (\mathcal{E}(\mathcal{M}),\mathcal{E})$ is an exact substructure that contains
$\mathcal{S}$. Let $0\rightarrow X^{\bullet}\rightarrow Y^{\bullet} \rightarrow Z^{\bullet}\rightarrow 0$ be a conflation in $(\mathcal{E}(\mathcal{M}),\mathcal{E}_{0})$. Then there is a conflation $0\rightarrow X^{\bullet}\rightarrow Q_0^{\bullet} \rightarrow Q_1^{\bullet}\rightarrow 0$ in $(\mathcal{E}(\mathcal{M}),\mathcal{E}_{0}^{1})\preceq (\mathcal{E}(\mathcal{M}),\mathcal{E}')$ by Proposition \ref{S(M) pre}. Consider the following commutative diagram

$$\xymatrix{
 	&&0\ar[d]&&0\ar[d]\\
 	0\ar[rr]&&X^{\bullet}\ar[d]^{f^{\bullet}}\ar[rr]^{r^{\bullet}}&&Q_0^{\bullet}\ar[d]^{s^{\bullet}}\ar[rr]&&Q_1^{\bullet}\ar@{=}[d]\ar[rr]&&0\\
 	0\ar[rr]&&Y^{\bullet}\ar[d]\ar[rr]^{t^{\bullet}}&&C^{\bullet}\ar[d]\ar[rr]&&Q_1^{\bullet}\ar[rr]&&0\\
 	&&Z^{\bullet}\ar@{=}[rr]\ar[d]&&Z^{\bullet}\ar[d]\\
    &&0&&0
 }$$
We have that $0\rightarrow Y^{\bullet}\rightarrow C^{\bullet} \rightarrow Q_1^{\bullet}\rightarrow 0$ is in $(\mathcal{E}(\mathcal{M}),\mathcal{E}_{0}^{1})$. Hence $0\rightarrow Q_0^{\bullet}\rightarrow C^{\bullet} \rightarrow Z^{\bullet}\rightarrow 0$ is in $(\mathcal{E}(\mathcal{M}),\mathcal{E}_{0})$. Since $Q_0^{\bullet}$ is in $\mathcal{S}(\mathcal{M})$, we have that $0\rightarrow Q_0^{\bullet}\rightarrow C^{\bullet} \rightarrow Z^{\bullet}\rightarrow 0$ is in $(\mathcal{E}(\mathcal{M}),\mathcal{E}_{0}^{-1})\preceq (\mathcal{E}(\mathcal{M}),\mathcal{E}')$. Now $s^{\bullet}r^{\bullet}=t^{\bullet}f^{\bullet}$ is a inflation in $(\mathcal{E}(\mathcal{M}),\mathcal{E}')$, hence $f^{\bullet}$ is a inflation in $(\mathcal{E}(\mathcal{M}),\mathcal{E}')$ by \cite[Proposition 2.16]{Bu}. This means $0\rightarrow X^{\bullet}\rightarrow Y^{\bullet} \rightarrow Z^{\bullet}\rightarrow 0$ is in $(\mathcal{E}(\mathcal{M}),\mathcal{E}')$.
\end{proof}

\begin{thm}\label{cluster quotients}
Let $(\mathcal{E},\mathcal{M})$ be an exact category. Then $\mathcal{E}(\mathcal{M})/\mathcal{S}(\mathcal{M})$ is an abelian category. Moreover, $(\mathcal{E}(\mathcal{M}),\mathcal{E}_0)$ is the unique exact structure such that it is a cluster quotient.
\end{thm}
\begin{proof}We first show that any conflation $0\rightarrow P^{\bullet}\rightarrow Y^{\bullet} \rightarrow Q^{\bullet}\rightarrow 0$ in $(\mathcal{E}(\mathcal{M}),\mathcal{E}_0)$ with $P^{\bullet},\ Q^{\bullet}$ in $\mathcal{S}(\mathcal{M})$ splits. This is clear, thus $\mathcal{S}(\mathcal{M})$ is a cluster tilting subcategory of $(\mathcal{E(M)},\mathcal{E}_0)$.

Suppose that $\mathcal{S(M)}$ is self-orthogonal with respect to some exact substructure $(\mathcal{E}(\mathcal{M}),\mathcal{E}')\preceq (\mathcal{E}(\mathcal{M}),\mathcal{E})$. If $(\mathcal{E}(\mathcal{M}),\mathcal{E}')\npreceq (\mathcal{E}(\mathcal{M}),\mathcal{E}_0)$, then there exists some conflation $0\rightarrow X^{\bullet}\rightarrow Y^{\bullet} \rightarrow Z^{\bullet}\rightarrow 0$ in $(\mathcal{E}(\mathcal{M}),\mathcal{E}')$ which is not splitting in degree 0. We can take a conflation $0\rightarrow X \to Q^{\bullet}_0 \rightarrow Q^{\bullet}_1\rightarrow 0$ in $(\mathcal{E}(\mathcal{M}),\mathcal{E}_{0}^{1})$ by Proposition \ref{S(M) pre}. Consider the following push-out diagram.

$$\xymatrix{
 	0\ar[rr]&&X^\bullet\ar[d]\ar[rr]&&Y^\bullet\ar[d]\ar[rr]&&Z^\bullet\ar@{=}[d]\ar[rr]&&0\\
 	0\ar[rr]&&Q_0^\bullet\ar[rr]&&M^\bullet\ar[rr]&&Z^\bullet\ar[rr]&&0\\
 }$$
Then $0\rightarrow Q_0^{\bullet}\rightarrow M^{\bullet} \rightarrow Z^{\bullet}\rightarrow 0$ is in $(\mathcal{E}(\mathcal{M}),\mathcal{E}')$ which is not splitting in degree 0. By the similar argument, we get a conflation $0\rightarrow Q_0^{\bullet}\rightarrow N^{\bullet} \rightarrow P_0^{\bullet}\rightarrow 0$ in $(\mathcal{E}(\mathcal{M}),\mathcal{E}')$ which is not splitting in degree 0. Now $P_0^{\bullet}$ and $Q_0^{\bullet} $ are in $\mathcal{S(M)}$, and $\mathcal{S(M)}$ is self-orthogonal with respect to $(\mathcal{E}(\mathcal{M}),\mathcal{E}')$, a contradiction.
Hence the conclusion holds by Theorem \ref{main Th}, Proposition \ref{S(M) pre} and Lemma \ref{min con s}.
\end{proof}

\bigskip {\bf Acknowledgements}
\bigskip

This research was partially supported by NSFC (Grant No. 11701455) and Fundamental Research Funds for the Central Universities (Grant No. 2452020182).

There are no relevant financial or non-financial competing interests to report.

\end{document}